\documentclass[11pt]{amsart}
\pdfoutput=1

\title{Universal end-compactifications of locally finite graphs}
\author{Jan Ouborny}
\author{Max Pitz}
\address{Universit\"at Hamburg, Department of Mathematics, Bundesstrasse 55 (Geomatikum), 20146 Hamburg, Germany}
\email{jan.ouborny@studium.uni-hamburg.de}
\email{max.pitz@uni-hamburg.de}
\usepackage{amsmath,amssymb,amsthm}  
\usepackage{mathtools}
\usepackage{tikz, tikz-cd}
\usetikzlibrary{shapes.geometric}
\usepackage{svg} 
\usepackage{comment}
\usepackage{enumerate}
\usepackage{enumitem}
\usepackage{mathrsfs}
\usepackage{todonotes}
\usepackage{xcolor} 	
\usepackage[unicode]{hyperref}
\hypersetup{
	colorlinks,
    linkcolor={red!60!black},
    citecolor={green!60!black},
    urlcolor={blue!60!black},
}
\usepackage[abbrev, msc-links]{amsrefs} 
\usepackage{caption}
\usepackage{subcaption}

\usepackage[utf8]{inputenc}
\usepackage[T1]{fontenc}
\usepackage{lmodern}
\usepackage[babel]{microtype}
\usepackage[english]{babel}

\usepackage{geometry}
\usepackage{enumitem}

\let\polishlcross=\l
\def\l{\ifmmode\ell\else\polishlcross\fi}

\let\emptyset=\varnothing

\let\theta=\vartheta
\let\rho=\varrho
\let\phi=\varphi

\def\NN{\mathbb N}

\def\cC{{\mathscr C}}

\def\cP{{\mathcal P}}

\defˆ{^}

\newcommand{\Set}[1]{{\left\lbrace {#1} \right\rbrace}}

\def\set#1:#2{\Set{{#1} \colon {#2}}}

\newcommand{\blowup}[1]{\mathbf{G}}

\DeclareFontFamily{U}  {MnSymbolC}{}
\DeclareSymbolFont{MnSyC}         {U}  {MnSymbolC}{m}{n}
\DeclareFontShape{U}{MnSymbolC}{m}{n}{
    <-6>  MnSymbolC5
   <6-7>  MnSymbolC6
   <7-8>  MnSymbolC7
   <8-9>  MnSymbolC8
   <9-10> MnSymbolC9
  <10-12> MnSymbolC10
  <12->   MnSymbolC12}{}
\DeclareMathSymbol{\powerset}{\mathord}{MnSyC}{180}

\theoremstyle{plain}
\newtheorem{thm}{Theorem}[section]

\newtheorem{prop}[thm]{Proposition}

\newtheorem{cor}[thm]{Corollary}
\newtheorem{lemma}[thm]{Lemma}

\theoremstyle{definition}

\begin{document}

\begin{abstract}
 We construct a locally finite connected graph whose Freudenthal compactification  is universal for the class of completely regular continua, a class also known in the literature under the name  thin or graph-like continua. 
\end{abstract}

\vspace*{-36pt}
\maketitle

\section{Introduction}

The purpose of this paper is twofold: First, to establish that the class $\mathcal{G}$ of \emph{completely regular} continua, also known in the literature as \emph{thin} or \emph{graph-like} continua, admits a universal object. This answers a recent question in the literature \cite{ABOBAKER2021107408}. Surprisingly, this universal object can already be found amongst the Freudenthal compactification of locally finite connected graphs. Over the previous two decades, the latter class has received considerable attention in graph theory, see for example the survey \cite{DSurv,DSurv2}, but not so much in topology. In this light, our second objective is to advertise the most important methods -- especially the various inverse limit methods -- underpinning the theory of Freudenthal compactifications of locally finite graphs, and demonstrate in which sense they may be used to obtain a complete understanding of the class $\mathcal{G}$ of topological continua mentioned above.

Recall that an object $X$ is \textit{universal} for a class of objects $\cP$ and an embedding relation $\leq$
if $X$ is contained in $\cP$ and every object $Y \in \cP$ can be embedded into $X$ with respect to $\leq$. 
The most important universal topological continua with respect to homeomorphic embeddings are the \emph{Hilbert cube} (universal for Peano continua \cite[1.4]{Nadler}), the \emph{Sierpiński curve} (universal for 1-dimensional planar continua \cite[1.11]{Nadler}), the \emph{Menger curve} (universal for  1-dimensional continua \cite{mayer1986menger}), or Wazewski's dendrite \cite[10.37 \& 10.49]{Nadler}  (universal for acyclic Peano continua).

In the field of graph theory the best known universal object is the Rado graph, which is universal for the class of countable graphs and the induced subgraph relation, see \cite[8.3.1]{diestel2015book} and \cite{rado1964universal}.
However, when sticking with fairly strong embedding relations such as the (induced) subgraph relation, then -- contrary to the situation in topology where universal objects abound -- further affirmative results for the existence of universal graphs are rather elusive: there is no subgraph-universal object for the class of locally finite connected graphs (by an argument attributed to de Bruijn in \cite{rado1964universal}), and 
responding to a question by Ulam, 
Pach showed that there is no subgraph-universal object for the class of countable planer graphs \cite{pach1981problem}. 

Thus, it makes sense to further weaken the graph embedding relation and search for universal graphs under topological embeddings. 
While Krill showed that there is still no universal planar graph under the topological embedding relation, \cite{krill2022universal},  the existence of a topologically universal locally finite connected graph now becomes a well-known exercise, and Lehner established in \cite{lehner2022universal} the existence of a topologically universal locally finite planar graph.

When thinking about locally finite graphs $G$ as topological objects, it has become apparent in the previous decade that in many situations it is advantageous to consider the Freudenthal compactification of $G$, denoted in this paper by $|G|$, see \cite{DSurv,DSurv2}. And so the natural question arises whether there are also universal such Freudenthal compactifications (for the topological embedding relation).

As our first result, we answer this question affirmatively by constructing a locally finite connected graph $\mathbf{G}$ such that every locally finite connected graph $G$ admits a topological embedding $G \hookrightarrow \mathbf{G}$ that extends to a topological embedding of $|G| \hookrightarrow |\mathbf{G}|$, see Theorem~\ref{thm_universal1}.

As our second result, we show in Theorem~\ref{thm_universalgraphlike} that this $|\mathbf{G}|$ is even universal for the much larger class of Peano continua known in the literature under various names such as \emph{graph-like continua}, \textit{completely regular continua} or \emph{thin continua}. This universality result answers a recent question of Abobaker and Charatonik \cite{ABOBAKER2021107408}, and also strengthens an earlier result by Espinoza et al.~\cite[Theorem~15]{espinoza2020} that every completely regular / graph-like continuum embeds into the Freudenthal compactification of \emph{some} locally finite connected graph. 

The key to our proof is a simplified inverse limit representation theorem for graph-like continua, where the factor spaces are finite graphs and where the bonding maps correspond to the contraction of one \emph{single} edge, Theorem~\ref{thm_invlimit}.

\section{Freudenthal compactifications of graphs}
For undefined graph theoretic terms we follow the terminology in \cite{diestel2015book}, and in particular \cite[Chapter~8]{diestel2015book} for ends in graphs; but see also \S\ref{sec_21} below. The symbol $[X]^k$ denotes the collection of $k$-element subsets of a set $X$.

Graphs can be viewed combinatorially as well as topologically, and we will move freely between both perspectives. Combinatorially, a graph $G=(V,E)$ consists of a (vertex-)set $V=V(G)$ and an (edge-)set $E=E(G)$ (where $V$ and $E$ are formally required to be disjoint) and a map $E \to [V]^1 \cup [V]^2$ assigning each edge $e$ either a single endvertex (in which case $e$ is a \emph{loop}) or a pair of endvertices. Note that our graphs may have parallel edges -- when we want to emphasize this, we also speak of \emph{multigraphs}. A graph $G$ is \emph{finite} if both $V(G)$ and $E(G)$ are finite, and it is \emph{locally finite} if every vertex in $V(G)$ is incident with only finitely many edges in $E(G)$.

To every combinatorial graph $G=(V,E)$ we can associate a topological graph, also denoted by $G$, by considering $G$ as the corresponding $1$-complex. If $G$ is finite, the corresponding topological space is compact, and we end up with a graph continuum in the sense that it is a Peano continuum with only finitely many branch points, cf.~\cite[\S IX]{Nadler}.

We write $K_n$ for the complete graph on $n$ vertices ($n \in \mathbb{N}$), i.e.\ the graph on $n$ vertices where any pair of vertices are connected by one edge. Given a graph $G$ and a set of vertices $U \subseteq V(G)$ we write $G[U]$ for the induced subgraph in the vertex set $U$.

For sets of vertices $A$ and $B$ (not necessarily disjoint) in a graph $G$, we define an \emph{$A-B$ path} in $G$ to be a path where one end vertex lies in $A$, the other in $B$ and which is otherwise disjoint from $A \cup B$. If $A$ and $B$ are disjoint, $E(A,B)$ denotes the set of edges in $G$ with one endpoint in $A$ and the other in $B$.

\subsection{Ends and end spaces}
\label{sec_21}
Given a locally finite, connected graph $G$, the space $|G|$ is the Freudenthal compactification of the (locally compact, $\sigma$-compact, metrizable) 1-complex $G$. However, when dealing with combinatorial aspects of these spaces, it is useful to have a more concrete, hands-on description of these spaces, and this is what we will discuss next (see also the textbook \cite[\S 8.6]{diestel2015book}).

A $1$-way infinite path is called a \emph{ray} and the subrays of a ray are its \emph{tails}. Two rays $R,R'$ in a graph $G = (V,E)$ are \emph{equivalent} if there exist infinitely many disjoint $R-R'$ paths between them in $G$; the corresponding equivalence classes of rays are the \emph{ends} of $G$. If $\omega$ is an end of $G$ and $R\in\omega$, we call $R$ an $\omega$-ray. The set of ends of a graph $G$ is denoted by $\Omega = \Omega(G)$.

If $X \subseteq V$ is finite and $\omega \in \Omega$, 
there is a unique component of $G-X$ that contains a tail of every $\omega$-ray. We denote this component by $C(X,\omega)$ and say that $\omega$ \emph{lives in $C(X,\omega)$}. We let $\Omega(X,\omega) = \set{\phi \in \Omega}:{C(X,\phi) = C(X,\omega)}$ denote the set of all ends that live in $C(X,\omega)$. We put $\hat{C}(X,\omega) = C(X,\omega) \cup \Omega(X,C)$.
The collection of singletons $\{v\}$ for $v \in V(G)$ together with all sets of the form $\hat{C}(X,\omega)$ for finite $X\subseteq V$ and $\omega \in \Omega(G)$ forms a basis for a topology on $V(G) \cup \Omega(G)$.
This topology is Hausdorff, and it is \emph{zero-dimensional} in the sense that it has a basis consisting of closed-and-open sets. 

We now describe the standard way to extend this topology on $V(G) \cup \Omega(G)$ to a topology on $|G| = G \cup \Omega(G)$, the graph $G$ together with its ends. 
This topology, called \textsc{Top}, has a basis formed by all open sets of $G$ considered as a $1$-complex, together with basic open neighbourhoods for ends of the form
\begin{align*}
    \hat{C}_*(X,\omega) := C(X,\omega) \cup \Omega(X,\omega) \cup E_*(X, C(X,\omega)),
\end{align*}
where $E_*(X, C(X,\omega))$ denotes any union of half-open intervals of all the edges from the edge cut $E(X, C(X,\omega))$ with endpoint in $C(X,\omega)$. For a proof of the following theorem see e.g.\ the textbook \cite[Theorem~8.6.1]{diestel2015book}.

\begin{thm}
\label{thm_freudenthal}
For every locally finite, connected graph $G$, the space $|G|$ is a Peano continuum, i.e.\ a compact, metrizable,  connected and locally connected space.
\end{thm}

If one starts with a summable length function on the edges of a locally finite, connected graph $G$, then $|G|$ will be homeomorphic to the metric completion of the length space $G$, see \cite{agelosedgelength}.

\subsection{An inverse limit description}

Given an infinite, locally finite connected graph $G=(V,E)$, let $S_0 \subset S_1 \subset S_2 \subset \cdots$ be a sequence of finite sets of vertices of $G$ such that $V(G) = \bigcup_{n \in \NN} S_n$. Write $G_n$ for the contraction of $G$ obtained by contracting each component of $G-S_n$ to a vertex, deleting any new loops that arise in the contraction but keeping parallel edges. The vertices of $G_n$ contracted from components of $G-S_n$ will be called the \textit{dummy vertices} of $G_n$. Let $\pi_n \colon G \to G_n$ denote the (continuous) contraction map.

Note that there are also natural contraction maps $f_n \colon G_{n+1} \to G_n$, which send the vertices of $G_{n+1}$ to the vertices of $G_n$ that contain them as subsets, which are the identity on the edges of $G_{n+1}$ that are also edges of $G_n$, and send any other edge of $G_{n+1}$ to the vertex of $G_n$ that contains both its endvertices. In other words, we get an inverse sequence $(G_n)_{n \in \NN}$ of finite (and hence compact) connected graphs, with onto bonding maps $f_n \colon G_{n+1} \to G_n$. 

\begin{thm}
\label{thm_limits}
For every locally finite, connected graph $G$, the space $|G|$ is homeomorphic to the inverse limit $\varprojlim G_n$.
\end{thm}

\begin{proof}
Every projection $\pi_n \colon G \to G_n$ extends to a continuous surjection $\tilde{\pi}_n \colon |G| \to G_n$, namely by sending every end $\omega$ of $|G|$ to the dummy vertex in $G_n$ corresponding to the component $C(S_n,\omega)$.
By definition of the basic open neighbourhoods of ends, this extension is continuous.
By standard inverse limit arguments,  see e.g.~\cite[2.22]{Nadler}, this yields a continuous surjection
$$ \pi \colon |G| \to \varprojlim G_n, \; x \mapsto (\tilde{\pi}_n(x))_{n \in \NN}.$$
As for any two distinct ends $\omega \neq \omega'$ of $G$ there is some $S_n$ with $C(S_n , \omega) \neq C(S_n, \omega')$ and hence $\tilde{\pi}_n ( \omega) \neq \tilde{\pi}_n ( \omega')$, we get that $\pi$ is injective. As a continuous bijection between compact Hausdorff spaces (Theorem~\ref{thm_freudenthal}), it follows that $\pi$ is a homeomorphism.
\end{proof}

\section{Topologically universal locally finite graphs}

\subsection{An example}

The existence of topologically universal locally finite connected graphs is well-known, see e.g.~\cite[Chapter~8, Exercise~49]{diestel2015book}. As a warm-up, we present yet another construction of topologically universal locally finite connected graph that we will later generalise.

Consider the disjoint union of complete graphs $K_n$ ($n \in \NN$), and insert all possible edges between any successive $K_n$ and $K_{n+1}$, see Figure~\ref{fig_example}. Call the resulting graph $\mathbf{G}$. 

\begin{figure}[h]
    \centering
    \begin{tikzpicture}[scale=1]
     \foreach \point in {
(0, 0.0),
(3, -0.3),
(3, 0.3),
(6, -0.6),
(6, 0.0),
(6, 0.6),
(9, -0.8999999999999999),
(9, -0.29999999999999993),
(9, 0.30000000000000004),
(9, 0.8999999999999999),
(12, -1.2),
(12, -0.6),
(12, 0.0),
(12, 0.5999999999999999),
(12, 1.2)
} {
            \def\this{vertex-\point}
            \node (\this) at \point  [circle,fill=black, inner sep=1.5pt] {} ;
        }
\draw[thick] (0, 0.0)--(3, -0.3);
\draw[thick] (0, 0.0)--(3, 0.3);
\draw[thick] (3, -0.3)--(6, -0.6);
\draw[thick] (3, -0.3)--(6, 0.0);
\draw[thick] (3, -0.3)--(6, 0.6);
\draw[thick] (3, 0.3)--(6, -0.6);
\draw[thick] (3, 0.3)--(6, 0.0);
\draw[thick] (3, 0.3)--(6, 0.6);
\draw[thick] (6, -0.6)--(9, -0.8999999999999999);
\draw[thick] (6, -0.6)--(9, -0.29999999999999993);
\draw[thick] (6, -0.6)--(9, 0.30000000000000004);
\draw[thick] (6, -0.6)--(9, 0.8999999999999999);
\draw[thick] (6, 0.0)--(9, -0.8999999999999999);
\draw[thick] (6, 0.0)--(9, -0.29999999999999993);
\draw[thick] (6, 0.0)--(9, 0.30000000000000004);
\draw[thick] (6, 0.0)--(9, 0.8999999999999999);
\draw[thick] (6, 0.6)--(9, -0.8999999999999999);
\draw[thick] (6, 0.6)--(9, -0.29999999999999993);
\draw[thick] (6, 0.6)--(9, 0.30000000000000004);
\draw[thick] (6, 0.6)--(9, 0.8999999999999999);
\draw[thick] (9, -0.8999999999999999)--(12, -1.2);
\draw[thick] (9, -0.8999999999999999)--(12, -0.6);
\draw[thick] (9, -0.8999999999999999)--(12, 0.0);
\draw[thick] (9, -0.8999999999999999)--(12, 0.5999999999999999);
\draw[thick] (9, -0.8999999999999999)--(12, 1.2);
\draw[thick] (9, -0.29999999999999993)--(12, -1.2);
\draw[thick] (9, -0.29999999999999993)--(12, -0.6);
\draw[thick] (9, -0.29999999999999993)--(12, 0.0);
\draw[thick] (9, -0.29999999999999993)--(12, 0.5999999999999999);
\draw[thick] (9, -0.29999999999999993)--(12, 1.2);
\draw[thick] (9, 0.30000000000000004)--(12, -1.2);
\draw[thick] (9, 0.30000000000000004)--(12, -0.6);
\draw[thick] (9, 0.30000000000000004)--(12, 0.0);
\draw[thick] (9, 0.30000000000000004)--(12, 0.5999999999999999);
\draw[thick] (9, 0.30000000000000004)--(12, 1.2);
\draw[thick] (9, 0.8999999999999999)--(12, -1.2);
\draw[thick] (9, 0.8999999999999999)--(12, -0.6);
\draw[thick] (9, 0.8999999999999999)--(12, 0.0);
\draw[thick] (9, 0.8999999999999999)--(12, 0.5999999999999999);
\draw[thick] (9, 0.8999999999999999)--(12, 1.2);

\draw (3, 0) ellipse (0.3cm and 0.6cm);
\draw (6, 0) ellipse (0.35cm and 0.8999999999999999cm);
\draw (9, 0) ellipse (0.4cm and 1.2cm);
\draw (12, 0) ellipse (0.5cm and 1.6cm);

\node () at (13, 0) {$\dots$} ;
\end{tikzpicture}
     \vspace{-1cm}
    \caption{A topologically universal locally finite connected graph (edges inside ellipses are omitted).}
    \label{fig_example}
\end{figure}
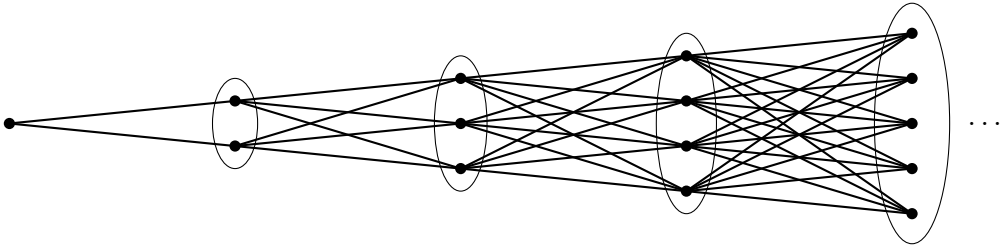

\begin{prop}
\label{prop_warum-up}
Every locally finite connected graph can be topologically embedded into $\mathbf{G}$.
\end{prop}

\begin{proof}
Let $G=(V,E)$ be a locally finite connected graph and let $v$ be a vertex of $G$.
We denote by $D^n$ the distance classes of $v$ in $G$, and by $D^{\leq n}$ we denote the union over the first $n$ distance classes.
First, we define a strictly increasing function $h \colon \NN \to \NN$ that determines in which $K_{h(n)}$ each $D_n$ will be embedded:
\[
    h(n) :=
   \begin{cases}
        \max\Set{h(n-1)+1, |E(D^n, D^{n+1})|, |D^{n}|} & \text{if } n > 0, \\
        d(v) & \text{if } n = 0. 
   \end{cases}
\]
We build the desired topological embedding $G \hookrightarrow \mathbf{G}$ by recursively constructing a sequence of topological embeddings 
$$\phi_n \colon D^{\leq n} \to \mathbf{G}\left[\bigcup_{k=1}^{h(n)} K_k\right]$$ for $n \in \NN$ such that
\begin{enumerate}
    \item \label{warmup_prop_extension} if $n \geq 1$, then $\phi_{n}$ extends $\phi_{n-1}$, i.e. $\phi_{n}\upharpoonright D^{\leq n-1} = \phi_{n-1}$, 
    \item \label{warumup_prop_distance_embed} every distance class $D^k$ with $k \leq n$ embeds under $\phi_n$ into $K_{h(k)}$, and
    \item \label{warmup_prop_subgraph_embed} 
    if $n \geq 1$, then the subgraph $G[D^{n-1} \cup D^{n}]$ embeds under $\phi_n$ into $\displaystyle \mathbf{G}\left[\bigcup^{h(n)}_{k=h(n-1)}K_k\right]$.
\end{enumerate}
To start our construction we embed $D^0=\Set{v}$ into $K_{h(0)}$.
Clearly, all properties are satisfied.

Assume that $\phi_n$ has been constructed for some $n \in \NN$. 
To construct $\phi_{n+1}$, we will extend $\phi_n$ as follows:
First, we embed $D^{n+1}$ into $K_{h(n+1)}$. 
Now we have to map each edge of $E(D^n, D^{n+1})$ to a corresponding $T^{h(n)}-T^{h(n+1)}$ path.
If $h(n+1)=h(n)+1$, it is clear how to map an edge of $E(D^n, D^{n+1})$.
Otherwise we first observe that for any $n < m$ the number of disjoint $K_n$--$K_m$ paths in $\mathbf{G}$ is $|K_n|$. 
Since $h(n) \geq |E(D^n, D^{n+1})|$, we get a set $\cP$ of $|E(D^n, D^{n+1})|$ disjoint $K_{h(n)+1}$--$K_{h(n+1)-1}$ paths. 
Then we choose pairwise distinct paths $P_e$ of $\cP$ for every $e=xy \in E(D^n, D^{n+1})$.
In order to obtain our desired paths, we extend each $P_e$ by one edge at each endvertex to connect it to $\phi_{n+1}(x)$ (resp. $\phi_{n+1}(y)$).
By property (3) of $\phi_n$, the paths intersect the image of $\phi_n$ only with their endpoints in $D^n$, and so $\phi_{n+1}$ is indeed a topological embedding.
This is possible, since each level of $\mathbf{G}$ is completely connected with every adjacent level. 
Clearly, properties (1) -- (3) are satisfied.

By property (\ref{warmup_prop_extension})
it is now clear that $\phi = \bigcup \phi_n \colon G \to \mathbf{G}$ is the desired embedding.
\end{proof}

\subsection{Blow-ups of trees}

A \emph{tree} is a connected, acyclic graph. The \emph{tree order} of a tree $T$ with root $r$ is a partial order $\leq_T$ on $V(T)$ defined by setting $t \leq_T s$ if $t$ lies on the unique path $rTs$ from $r$ to $s$ in $T$. Given $n \in \mathbb{N}$, the \emph{$n$th level} $T^n$ of $T$ is the set of vertices at distance $n$ from $r$ in $T$, and by $T^{\leq n}$ we denote the union over the first $n$ levels. The set $T^{<n}$ is defined in analogy. The \emph{down-closure} of a vertex $t$ is the set $\lceil t \rceil := \{\,s \in T \colon s \leq t\,\}$; its \emph{up-closure} is the set $\lfloor t \rfloor := \{\,s \in T \colon t \leq s\,\}$. The down-closure of $t$ is always a finite chain, the vertex set of the path $rTt$. A ray $R \subseteq T$ starting at the root is called a \emph{normal ray} of $T$.

Let $t$ be a node of $T$. The \emph{children} of $t$ are the neighbours of $t$ that belong to $\lfloor t \rfloor$, and if $t \neq r$ we call the unique neighbour of $t$ in $\lceil t \rceil$ its \textit{parent}.

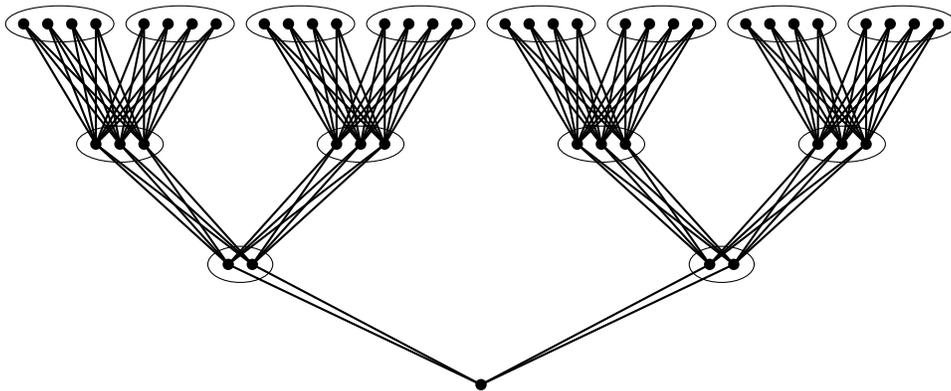
\begin{figure}[h]
    \centering
    \begin{tikzpicture}[scale=0.8, transform shape]
     \foreach \point in {
	(4.0, 0),
	(7.8, 2),
	(8.200000000000001, 2),
	(9.6, 4),
	(10.0, 4),
	(10.4, 4),
	(10.4, 6),
	(10.8, 6),
	(11.200000000000001, 6),
	(11.6, 6),
	(8.4, 6),
	(8.8, 6),
	(9.200000000000001, 6),
	(9.6, 6),
	(5.6, 4),
	(6.0, 4),
	(6.3999999999999995, 4),
	(6.4, 6),
	(6.800000000000001, 6),
	(7.199999999999999, 6),
	(7.6, 6),
	(4.4, 6),
	(4.800000000000001, 6),
	(5.199999999999999, 6),
	(5.6, 6),
	(-0.2, 2),
	(0.2, 2),
	(1.6, 4),
	(2.0, 4),
	(2.4, 4),
	(2.4, 6),
	(2.8, 6),
	(3.1999999999999997, 6),
	(3.6, 6),
	(0.3999999999999999, 6),
	(0.7999999999999998, 6),
	(1.2, 6),
	(1.6, 6),
	(-2.4, 4),
	(-2.0, 4),
	(-1.6, 4),
	(-1.6, 6),
	(-1.2000000000000002, 6),
	(-0.8, 6),
	(-0.3999999999999999, 6),
	(-3.6, 6),
	(-3.2, 6),
	(-2.8000000000000003, 6),
	(-2.4, 6)
} {
            \def\this{vertex-\point}
            \node (\this) at \point  [circle,fill=black, inner sep=1.9pt] {} ;
        }
\draw (8.0, 2) ellipse (0.54cm and 0.3cm);
\draw (10.0, 4) ellipse (0.72cm and 0.3cm);
\draw (11.0, 6) ellipse (0.8999999999999999cm and 0.3cm);
\draw (9.0, 6) ellipse (0.8999999999999999cm and 0.3cm);
\draw (6.0, 4) ellipse (0.72cm and 0.3cm);
\draw (7.0, 6) ellipse (0.8999999999999999cm and 0.3cm);
\draw (5.0, 6) ellipse (0.8999999999999999cm and 0.3cm);
\draw (0.0, 2) ellipse (0.54cm and 0.3cm);
\draw (2.0, 4) ellipse (0.72cm and 0.3cm);
\draw (3.0, 6) ellipse (0.8999999999999999cm and 0.3cm);
\draw (1.0, 6) ellipse (0.8999999999999999cm and 0.3cm);
\draw (-2.0, 4) ellipse (0.72cm and 0.3cm);
\draw (-1.0, 6) ellipse (0.8999999999999999cm and 0.3cm);
\draw (-3.0, 6) ellipse (0.8999999999999999cm and 0.3cm);
\draw[thick] (4.0, 0)--(7.8, 2);
\draw[thick] (4.0, 0)--(8.200000000000001, 2);
\draw[thick] (7.8, 2)--(9.6, 4);
\draw[thick] (8.200000000000001, 2)--(9.6, 4);
\draw[thick] (7.8, 2)--(10.0, 4);
\draw[thick] (8.200000000000001, 2)--(10.0, 4);
\draw[thick] (7.8, 2)--(10.4, 4);
\draw[thick] (8.200000000000001, 2)--(10.4, 4);
\draw[thick] (9.6, 4)--(10.4, 6);
\draw[thick] (10.0, 4)--(10.4, 6);
\draw[thick] (10.4, 4)--(10.4, 6);
\draw[thick] (9.6, 4)--(10.8, 6);
\draw[thick] (10.0, 4)--(10.8, 6);
\draw[thick] (10.4, 4)--(10.8, 6);
\draw[thick] (9.6, 4)--(11.200000000000001, 6);
\draw[thick] (10.0, 4)--(11.200000000000001, 6);
\draw[thick] (10.4, 4)--(11.200000000000001, 6);
\draw[thick] (9.6, 4)--(11.6, 6);
\draw[thick] (10.0, 4)--(11.6, 6);
\draw[thick] (10.4, 4)--(11.6, 6);
\draw[thick] (9.6, 4)--(8.4, 6);
\draw[thick] (10.0, 4)--(8.4, 6);
\draw[thick] (10.4, 4)--(8.4, 6);
\draw[thick] (9.6, 4)--(8.8, 6);
\draw[thick] (10.0, 4)--(8.8, 6);
\draw[thick] (10.4, 4)--(8.8, 6);
\draw[thick] (9.6, 4)--(9.200000000000001, 6);
\draw[thick] (10.0, 4)--(9.200000000000001, 6);
\draw[thick] (10.4, 4)--(9.200000000000001, 6);
\draw[thick] (9.6, 4)--(9.6, 6);
\draw[thick] (10.0, 4)--(9.6, 6);
\draw[thick] (10.4, 4)--(9.6, 6);
\draw[thick] (7.8, 2)--(5.6, 4);
\draw[thick] (8.200000000000001, 2)--(5.6, 4);
\draw[thick] (7.8, 2)--(6.0, 4);
\draw[thick] (8.200000000000001, 2)--(6.0, 4);
\draw[thick] (7.8, 2)--(6.3999999999999995, 4);
\draw[thick] (8.200000000000001, 2)--(6.3999999999999995, 4);
\draw[thick] (5.6, 4)--(6.4, 6);
\draw[thick] (6.0, 4)--(6.4, 6);
\draw[thick] (6.3999999999999995, 4)--(6.4, 6);
\draw[thick] (5.6, 4)--(6.800000000000001, 6);
\draw[thick] (6.0, 4)--(6.800000000000001, 6);
\draw[thick] (6.3999999999999995, 4)--(6.800000000000001, 6);
\draw[thick] (5.6, 4)--(7.199999999999999, 6);
\draw[thick] (6.0, 4)--(7.199999999999999, 6);
\draw[thick] (6.3999999999999995, 4)--(7.199999999999999, 6);
\draw[thick] (5.6, 4)--(7.6, 6);
\draw[thick] (6.0, 4)--(7.6, 6);
\draw[thick] (6.3999999999999995, 4)--(7.6, 6);
\draw[thick] (5.6, 4)--(4.4, 6);
\draw[thick] (6.0, 4)--(4.4, 6);
\draw[thick] (6.3999999999999995, 4)--(4.4, 6);
\draw[thick] (5.6, 4)--(4.800000000000001, 6);
\draw[thick] (6.0, 4)--(4.800000000000001, 6);
\draw[thick] (6.3999999999999995, 4)--(4.800000000000001, 6);
\draw[thick] (5.6, 4)--(5.199999999999999, 6);
\draw[thick] (6.0, 4)--(5.199999999999999, 6);
\draw[thick] (6.3999999999999995, 4)--(5.199999999999999, 6);
\draw[thick] (5.6, 4)--(5.6, 6);
\draw[thick] (6.0, 4)--(5.6, 6);
\draw[thick] (6.3999999999999995, 4)--(5.6, 6);
\draw[thick] (4.0, 0)--(-0.2, 2);
\draw[thick] (4.0, 0)--(0.2, 2);
\draw[thick] (-0.2, 2)--(1.6, 4);
\draw[thick] (0.2, 2)--(1.6, 4);
\draw[thick] (-0.2, 2)--(2.0, 4);
\draw[thick] (0.2, 2)--(2.0, 4);
\draw[thick] (-0.2, 2)--(2.4, 4);
\draw[thick] (0.2, 2)--(2.4, 4);
\draw[thick] (1.6, 4)--(2.4, 6);
\draw[thick] (2.0, 4)--(2.4, 6);
\draw[thick] (2.4, 4)--(2.4, 6);
\draw[thick] (1.6, 4)--(2.8, 6);
\draw[thick] (2.0, 4)--(2.8, 6);
\draw[thick] (2.4, 4)--(2.8, 6);
\draw[thick] (1.6, 4)--(3.1999999999999997, 6);
\draw[thick] (2.0, 4)--(3.1999999999999997, 6);
\draw[thick] (2.4, 4)--(3.1999999999999997, 6);
\draw[thick] (1.6, 4)--(3.6, 6);
\draw[thick] (2.0, 4)--(3.6, 6);
\draw[thick] (2.4, 4)--(3.6, 6);
\draw[thick] (1.6, 4)--(0.3999999999999999, 6);
\draw[thick] (2.0, 4)--(0.3999999999999999, 6);
\draw[thick] (2.4, 4)--(0.3999999999999999, 6);
\draw[thick] (1.6, 4)--(0.7999999999999998, 6);
\draw[thick] (2.0, 4)--(0.7999999999999998, 6);
\draw[thick] (2.4, 4)--(0.7999999999999998, 6);
\draw[thick] (1.6, 4)--(1.2, 6);
\draw[thick] (2.0, 4)--(1.2, 6);
\draw[thick] (2.4, 4)--(1.2, 6);
\draw[thick] (1.6, 4)--(1.6, 6);
\draw[thick] (2.0, 4)--(1.6, 6);
\draw[thick] (2.4, 4)--(1.6, 6);
\draw[thick] (-0.2, 2)--(-2.4, 4);
\draw[thick] (0.2, 2)--(-2.4, 4);
\draw[thick] (-0.2, 2)--(-2.0, 4);
\draw[thick] (0.2, 2)--(-2.0, 4);
\draw[thick] (-0.2, 2)--(-1.6, 4);
\draw[thick] (0.2, 2)--(-1.6, 4);
\draw[thick] (-2.4, 4)--(-1.6, 6);
\draw[thick] (-2.0, 4)--(-1.6, 6);
\draw[thick] (-1.6, 4)--(-1.6, 6);
\draw[thick] (-2.4, 4)--(-1.2000000000000002, 6);
\draw[thick] (-2.0, 4)--(-1.2000000000000002, 6);
\draw[thick] (-1.6, 4)--(-1.2000000000000002, 6);
\draw[thick] (-2.4, 4)--(-0.8, 6);
\draw[thick] (-2.0, 4)--(-0.8, 6);
\draw[thick] (-1.6, 4)--(-0.8, 6);
\draw[thick] (-2.4, 4)--(-0.3999999999999999, 6);
\draw[thick] (-2.0, 4)--(-0.3999999999999999, 6);
\draw[thick] (-1.6, 4)--(-0.3999999999999999, 6);
\draw[thick] (-2.4, 4)--(-3.6, 6);
\draw[thick] (-2.0, 4)--(-3.6, 6);
\draw[thick] (-1.6, 4)--(-3.6, 6);
\draw[thick] (-2.4, 4)--(-3.2, 6);
\draw[thick] (-2.0, 4)--(-3.2, 6);
\draw[thick] (-1.6, 4)--(-3.2, 6);
\draw[thick] (-2.4, 4)--(-2.8000000000000003, 6);
\draw[thick] (-2.0, 4)--(-2.8000000000000003, 6);
\draw[thick] (-1.6, 4)--(-2.8000000000000003, 6);
\draw[thick] (-2.4, 4)--(-2.4, 6);
\draw[thick] (-2.0, 4)--(-2.4, 6);
\draw[thick] (-1.6, 4)--(-2.4, 6);
\end{tikzpicture}
    \vspace{-1cm}
    \caption{A blowup of the binary tree (edges inside ellipses are omitted). }
    \label{fig_blowup}
\end{figure}

To construct our desired topologically universal graph $\mathbf{G}$, we will replace vertices of trees by copies of complete graphs such that we increase connectivity in the upper levels (but also keep the tree's end structure).
Let $T$ be a rooted tree. 
Now we obtain our desired graph $\blowup{T}$ by replacing every vertex $t \in T^n$ by a copy $K(t)$ of $K_{n+1}$, a complete graph on ${n+1}$ vertices, and joining each $K(t)$ completely to all $K(t')$ with $tt' \in E(T)$, see also Figure~\ref{fig_blowup}.
Formally, $\mathbf{G}$ is the graph on
\[
 \bigcup_{n \geq 0} (T^n \times K_{n+1})
\]
in which two vertices $(t, v)$ and $(t', v')$ are adjacent if and only if $tt' \in E(T)$, or $t=t'$. 
For a subset $S \subseteq T$ we write $\mathbf{G}[S]:=\mathbf{G}[\bigcup_{s \in S} K(s)]$, and 
for a vertex $t \in T$, we  write $\blowup{T}_{\lfloor t \rfloor }$ short for $\mathbf{G}[\lfloor t \rfloor ]$.

For later use, we isolate the following simple connectivity property of $\mathbf{G}$. Let us call a path $P \subseteq \mathbf{G}$ \emph{monotone} if $P$ intersects every \emph{level} $\mathbf{G}[T^{n}]$ of $\mathbf{G}$ in at most one point.

\begin{lemma}
\label{lem_connectivity}
If $X \subseteq \mathbf{G}[T^{\geq n}]$ is a set of vertices such that $|X \cap \mathbf{G}[T^k]| \leq n$ for all $k \geq n$, then there exist monotone paths between any two vertices of distinct $K(t)$ and $K(t')$ with $t \leq_T t'$ in $\mathbf{G} - X$.
\end{lemma}

\begin{proof}
The assumptions on $X$ implies that $K(t) -X$ is non-empty for all $t \in T$, and so for $tt' \in E(T)$, any vertex of $K(t)-X$ is connected to any vertex in $K(t')-X$. 
\end{proof}

Let $b \colon \mathbf{G} \to T$ be the contraction map that collapses each $K(t)$ back to $t$. 
As $b$ has finite/compact fibres, i.e.\ $b$ is a continuous, perfect mapping in topological terms, it is well-known that $b$ extends to a continuous map between the Freudenthal boundaries $\Omega(G)$ and $\Omega(T)$. In the present case, this extension is actually a homeomorphism, as we will now show.

\begin{prop}
The endspaces $\Omega(\mathbf{G})$ and $\Omega(T)$ are naturally homeomorphic in the sense that the contraction map $b \colon \mathbf{G} \to T$ extends to a continuous surjection $\tilde{b} \colon |\mathbf{G}| \to |T|$ inducing a homeomorphism $ \Omega(\mathbf{G}) \to \Omega(G)$.
\end{prop}

\begin{proof}
First, let us represent $|T|$ and $|\mathbf{G}|$ as inverse limits. For this, consider the sequences $S_n = T^{<n} \subset V(T)$ and $S'_n = b^{-1}(S_n) \subset V(\mathbf{G})$, and consider the contractions $T_n$ and $\mathbf{G}_n$ obtained by contracting all components of $T-S_n$ and $\mathbf{G} - S'_n$ respectively. By Theorem~\ref{thm_limits} we have 
$$|T| = \varprojlim T_n \quad \text{and} \quad |\mathbf{G}| = \varprojlim \mathbf{G}_n.$$
Moreover, the continuous surjection $b \colon \mathbf{G} \to T$ naturally restricts to a surjection
$b_n \colon \mathbf{G}_n \to T_n$ 
which by choice of $S'_n = b^{-1}(S_n)$ induces a bijection between the dummy vertices of $\mathbf{G}_n$ and $T_n$ $(\star)$.
We claim that
$$ \tilde{b} \colon \varprojlim \mathbf{G}_n \to \varprojlim T_n, \; (x_n)_{n \in \NN} \to (b_n(x_n))_{n \in \NN}$$
is as desired. Clearly, it is a well-defined continuous map. As all $b_n$ are surjections, so is $\tilde{b}$ by compactness. It remains to argue that $\tilde{b}$ sends ends of $\mathbf{G}$ injectively to ends of $T$. 
So consider two distinct ends $\omega \neq \omega' \in \Omega(\mathbf{G})$. 
In $\varprojlim \mathbf{G}_n = |\mathbf{G}| $, the ends $\omega$ and $\omega'$ are represented by sequences of dummy vertices $(x_n)_{n \in \NN}$ and $(x'_n)_{n \in \NN}$ respectively. Since $\omega \neq \omega'$ we have $x_n \neq x'_n$ for some $n \in \NN$. But then
$b_n(x_n) \neq b_n(x'_n)$ by $(\star)$, so $\tilde{b}$ is injective.
\end{proof}

\subsection{A topologically universal $|G|$}

In this section we show that our desired topologically universal end compactification of a locally finite graph is given by the blowup $\mathbf{G}$ of the full binary tree.
From now on we will denote the full binary tree with $T$.

\begin{thm}
\label{thm_universal1}
Every locally finite connected graph $G$ can be embedded into $\mathbf{G}$ such that the embedding $\phi \colon G \to \mathbf{G}$ lifts to an embedding $\tilde{\phi} \colon |G| \to |\mathbf{G}|$.
\end{thm}

\begin{proof}[Proof of Theorem~\ref{thm_universal1}]
Let $G=(V,E)$ be a locally finite connected graph and $v \in V$ be an arbitrary vertex.
We denote by $D^n$ the distance classes of $v$ in $G$, and by $D^{\leq n}$ we denote the union over the first $n$ distance classes. Write
$$\cC_{n+1}:=\Set{ C \cap D^{n+1} \colon C 
\text{ a component of } G-D^{\leq n}}.$$
Any subgraph $H \in \cC_{n+1}$ has neighbours in precisely one subgraph from $\cC_n$, which we denote with $H^-$.
First, we define a strictly increasing function $h \colon \NN \to \NN$ which determines at which level of $\mathbf{G}$ each $D^n$ will be embedded:
\[
    h(n) :=
     \begin{cases}
        \max\Set{h(n-1)+|D^n|, |E(D^n, D^{n+1})|} & \text{if } n > 0, \\
        d(v) & \text{if } n = 0.
    \end{cases}
\]

We build the desired topological embedding $G \hookrightarrow \mathbf{G}$ by recursively constructing a sequence of topological embeddings 
$$\phi_n \colon D^{\leq n} \to \mathbf{G}\left[T^{\leq h(n)}\right]$$ for $n \in \NN$ such that
\begin{enumerate}
    \item \label{univ1_extend} if $n \geq 1$, then $\phi_{n}$ extends $\phi_{n-1}$, i.e. $\phi_{n}\upharpoonright D^{\leq n-1} = \phi_{n-1}$, and 
    \item \label{univ1_prop_subgraph} the subgraphs $H \in \cC_{n}$ embed into pairwise distinct $K(t)$ for $t \in T^{h(n)}$, and 
     \item \label{univ1_prop_phi_cross_edges} if $n \geq 1$, then the subgraph $G[D^{n-1} \cup D^{n}]$ embeds under $\phi_n$ into $\displaystyle \mathbf{G}\left[\bigcup^{h(n)}_{k=h(n-1)}T^k\right]$.
\end{enumerate}
Embed $v$ into an arbitrary $K(t)$ with $t \in T^{h(0)}$. 
Clearly, all properties are satisfied. 

Now suppose that $\phi_n$ has been constructed for some $n \in \NN$.
To extend $\phi_n$ to $\phi_{n+1}$, we start by embedding each subgraph $H \in \cC_{n+1}$.
By property (\ref{univ1_prop_subgraph}) of $\phi_n$, every subgraph $H^- \in \cC_{n}$ of any $H \in \cC_{n+1}$ gets embedded by $\phi_n$ into a  $K(t)$ for distinct $t \in T^{h(n)}$.
Every blowup vertex of $T^{h(n+1)}$ has at least an order of $h(n+1) \geq |D^{n+1}|\geq |H|$.
Hence, we can embed each $H$ into any blowup vertex $K(s)$ with $s \in T^{h(n+1)} \cap T_{\lfloor t \rfloor}$.
Since $h(n+1) \geq h(n) + |D^{n+1}|$, we can choose the $K(s)$ such that they are pairwise distinct and thus property (\ref{univ1_prop_subgraph}) is satisfied.
It remains to map the edges $E(D^n, D^{n+1})$ to their corresponding paths. 
Enumerate 
$$E(D^n,D^{n+1})= \{x_iy_i \colon 1 \leq i \leq k(n) \}$$
and note that $k(n) \leq h(n)$. We recursively choose monotone, pairwise independent $\phi_{n+1}(x_i)-\phi_{n+1}(y_i)$ paths $P_i$ as the image of the edge $x_iy_i$ under $\phi_{n+1}$: If $1 \leq j < k(n)$ is such that paths $P_i$ for $i<j$ have already been defined, consider $X_j = \{ \mathring{P}_i \colon i < j\}$, the union over all inner vertices from all previous $P_i$. Then $X_j \subseteq \mathbf{G}[T^{\geq h(n)+1}]$ satisfies the assumptions of Lemma~\ref{lem_connectivity}, and any monotone $\phi_{n+1}(x_j)-\phi_{n+1}(y_j)$ path $P_j$ in $\mathbf{G} - X_j$ is as desired.
Since all paths $P_i$ are monotone, it is clear that $\phi_{n+1}$ satisfies property (\ref{univ1_prop_phi_cross_edges}) for $\phi_{n+1}$.
Furthermore, by property (\ref{univ1_prop_phi_cross_edges}) of $\phi_n$ these paths intersect the image of $\phi_n$ only with their endpoints in $H'$, and so $\phi_{n+1}$ is indeed a topological embedding.

By (\ref{univ1_extend}), the union of the $\phi_n$ is a topological embedding $\phi \colon G \hookrightarrow \mathbf{G}$.
We claim that properties (\ref{univ1_prop_subgraph})  and (\ref{univ1_prop_phi_cross_edges}) yield the following property of $\phi$ for every $n \in \NN$:
\begin{enumerate}
\item[$(\ast)$]\label{prop_comp_subset} Distinct components of $ G-D^{<n}$ map under $\phi$ into  distinct components of $\mathbf{G}-\mathbf{G}[T^{<h(n)}]$.
\end{enumerate}
Indeed, let $C$ be a component of $G-D^{<n}$.
Property (\ref{univ1_prop_phi_cross_edges}) implies that $C$ gets mapped under $\phi$ into a connected component of 
$$\displaystyle \mathbf{G}\left[\bigcup_{k \geq h(n)}T^k\right] = \mathbf{G}-\mathbf{G}[T^{<h(n)}],$$
and any such component is of the form $\blowup{T}_{\lfloor t_C \rfloor }$ for some $t_C \in T^{h(n)}$.
Furthermore, property (\ref{univ1_prop_subgraph}) ensures that for any two distinct components $C \neq C'$ we have $t_C \neq t_{C'}$, establishing $(\ast)$.

We now show that every homeomorphic embedding $\phi \colon G \to \mathbf{G}$ with property $(\ast)$ extends to an embedding $\tilde{\phi} \colon |G| \hookrightarrow \mathbf{G}$.
To obtain $\tilde{\phi}$, we represent $|\mathbf{G}| = \varprojlim \mathbf{G}_n$ as an inverse limit by considering contractions $\mathbf{G}_n$ obtained by contracting all components of $\mathbf{G}-\mathbf{G}[T^{\leq h(n)}]$ (Theorem~\ref{thm_limits}).  Write $\pi_n \colon \mathbf{G} \to \mathbf{G}_n$ for the corresponding contraction maps.
Now every map $\pi_n \circ \phi \colon G \to \mathbf{G}_n$ extends to a continuous map $\tilde{\phi}_n \colon |G| \to \mathbf{G}_n$ by sending each end $\omega$ of $G$ to the dummy vertex of $\mathbf{G}_n$ containing $\phi(C(D^{<n}, \omega))$.
This yields a continuous map
\[
    \tilde{\phi} \colon |G| \to \varprojlim \mathbf{G}_n, x \mapsto (\tilde{\phi}_n(x))_{n \in \NN}.
\]
Finally, to see that $\tilde{\phi}$ is injective, 
we consider two distinct ends $\omega \neq \omega'$ of $G$. 
There exists some $n \in \NN$ such that $C(D^{<n}, \omega) \neq C(D^{<n}, \omega')$.
These are components of $G-D^{<n}$ and thus by property $(\ast)$ we have $\tilde{\phi}_n(\omega) \neq \tilde{\phi}_n(\omega')$.
\end{proof}

\section{Universality for graph-like continua}
\subsection{Graph-like continua and inverse limits}

Taking our cue from \cite{thomassenvella} we call a compact metrizable space $\textit{graph-like}$ if there is a closed zero-dimensional subset $V \subseteq X$ and  a discrete index set $E$ such that $X \setminus V \cong E \times (0,1)$. The points in $V$ are the \emph{vertices} and the elements of $E$ the edges of $X$. Combinatorially, if we have one fixed set $V$ in mind, we also write $(X,V,E)$ as a triple. Note that every finite graph and every Freudenthal compactification of a locally finite connected graph is a graph-like space (where the vertices plus the ends take the role of the closed, zero-dimensional subset). Further examples are given by the Hawaiian earring and a version of Sierpiński's triangle, where one iteratively removes open hexagons instead of open triangles, see Figure~\ref{fig_hawplusSier}.
For further examples and background information on graph-like continua, see \cite{espinoza2020,GartsidePitzgraphlike,Eulerspaces}.

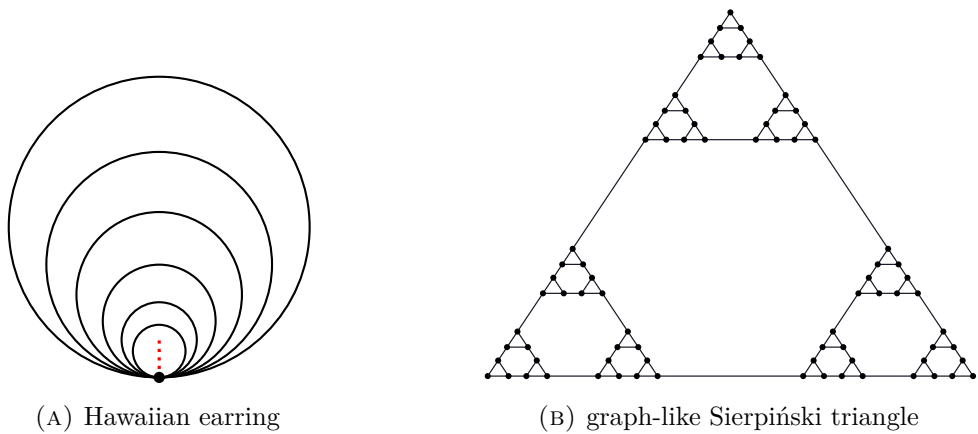
\begin{figure}[h]
     \centering
     \begin{subfigure}[b]{0.4\textwidth}
        \centering
        \begin{tikzpicture}
            \draw[thick] (0,0) arc [start angle=-90, end angle=270, radius=2];
            \draw[thick] (0,0) arc [start angle=-90, end angle=270, radius=1.5];
            \draw[thick] (0,0) arc [start angle=-90, end angle=270, radius=1.1];
            \draw[thick] (0,0) arc [start angle=-90, end angle=270, radius=0.75];
            \draw[thick] (0,0) arc [start angle=-90, end angle=270, radius=0.5];
            \draw[thick] (0,0) arc [start angle=-90, end angle=270, radius=0.35];
            
            \draw[dotted,red,very thick] (0,0) -- (0,0.5);
            
            \node at (0,0)  [circle,fill=black, inner sep=1.5pt] {} ;
        \end{tikzpicture}

        \caption{Hawaiian earring}
     \end{subfigure}
     \begin{subfigure}[b]{0.5\textwidth}
         \centering
    \includegraphics[width=0.8\textwidth]{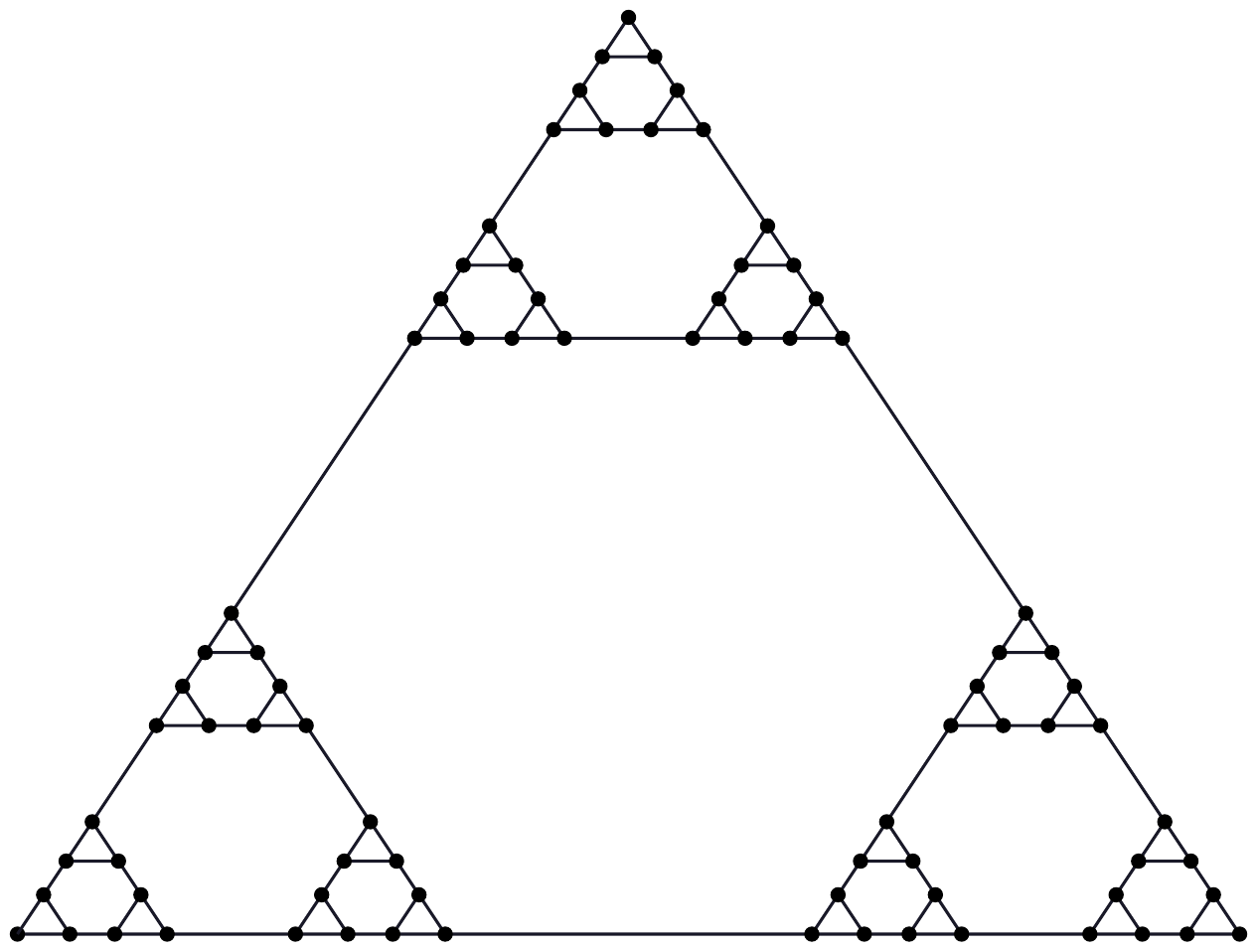}
         \caption{graph-like Sierpiński triangle}
     \end{subfigure}
     \caption{Two examples of graph-like continua.}
     \label{fig_hawplusSier}
\end{figure}

 We often identify the label, $e$, of an edge, with the (open) subspace $e\times (0,1)$ of $X$.
Since $V$ is zero-dimensional, for every edge $e$, the closure, $\overline{e}$, of $e$ adds at most two vertices -- the ends of the edge -- and $\overline{e}$ is a simple closed curve or an arc. Since compact metrizable spaces are separable, every graph-like continuum has at most countably many edges.

Every graph-like continuum is locally connected, and hence a Peano continuum \cite[Corollary 8]{espinoza2020}. The class of graph-like continua has been studied under a number of different names.
In continuum theory they are known as completely regular continua: Recall that a continuum $X$ is called \textit{completely regular continuum} if each non-degenerate subcontinuum of $X$ has non-empty interior in $X$. That this condition is in fact equivalent to being a graph-like continuum has been proven already in \cite{kra}.
Unaware of the notions of graph-like continua and completely regular continua, Abobaker and Charatonik recently introduced in \cite{ABOBAKER2021107408} the equivalent notion of a \emph{thin continuum}, and posed the question about the existence of a universal thin continuum. For the equivalences, also see \cite[Theorem A]{espinoza2020}.
Our universality theorem for graph-like continua relies on the following inverse limit representation theorem.

\begin{thm}
\label{thm_invlimit}
Every graph-like continuum $X$ can be represented as inverse limit of finite connected multigraphs $\set{G_n}:{n \in \NN}$ where the bonding maps $f_n \colon G_{n+1} \to G_n$ correspond to the contraction of one single edge, i.e.\ we have $G_{n+1} / \overline{e}_{n+1} = G_n$.
\end{thm}

\begin{proof}
Let $(X,V,E)$ be a graph-like continuum with vertex set $V$ and edge set $E$. Enumerate $E = \{e_1,e_2,e_3,\ldots\}$ and for $n=0,1,2,\ldots$ let $G_n$ be the quotient of $X$ given by contracting each connected component of $X-\{e_1,\ldots,e_n\}$ to a point, and let $\pi_n \colon X \to G_n$ be the (continuous) quotient map. 
Since deleting an edge from a graph-like continuum leaves at most two connected components (which can be seen for example by invoking the \emph{boundary bumping theorem} \cite[5.6]{Nadler}), every $G_n$ is a finite multigraph with at most $n+1$ vertices and edge set $\{e_1,\ldots,e_n\}$. As the quotient of a connected space, every $G_n$ is connected, too. See Figure~\ref{fig_invlimit} for an example in the case of the graph-like Sierpiński triangle.

\begin{figure}[h]
    \centering
    \includegraphics[width=\textwidth]{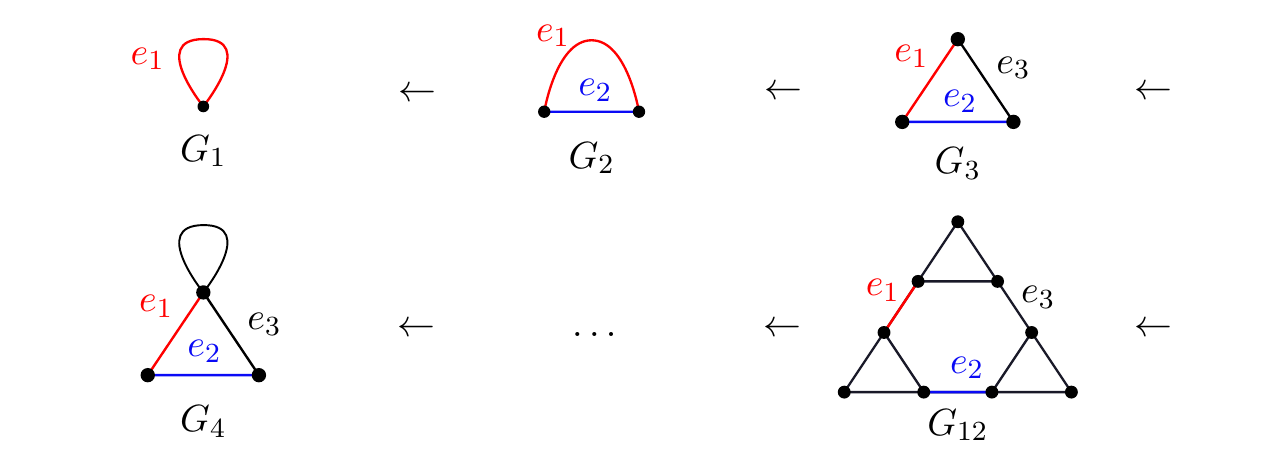}
    \caption{Representing the graph-like Sierpiński triangle as inverse limit with single edge contraction bonding maps.}
    \label{fig_invlimit}
\end{figure}

From the construction, it is evident that $G_{n+1} / \overline{e}_{n+1} = G_n$. Let $f_{n} \colon G_{n+1} \to G_n$ be the corresponding quotient map.
We claim that $X$ is represented by the inverse limit
$$\varprojlim \set{(G_n,f_n)}:{n \in \NN}.$$
To see this claim, consider the map 
$$f \colon X \to  \varprojlim G_n, \; x \mapsto (\pi_n(x))_{n \in \NN}.$$
 Since $X$ is compact and $\varprojlim G_n$ is Hausdorff, it suffices to show that $f$ is a continuous bijection. By the product topology, $f$ is a continuous map into the product $\prod_{n \in \NN} G_n$. Moreover, it is straightforward from the definition of $f_{n}$ to check that the image of $f$ is a subset of $\varprojlim G_n$. That $f$ is surjective follows from the fact that each $\pi_n$ is a continuous surjection and $X$ is compact, see e.g.~\cite[2.22]{Nadler}. Finally, to see that $f$ is injective, recall that any two vertices $x \neq y$ of $X$ are separated by deleting finitely many edges \cite[Lemma~1]{espinoza2020}, and so $f_n(x) \neq f_n(y)$ as soon as these finitely many edges are contained in $ \{e_1,\ldots,e_n\}$.
\end{proof}

We remark that the graph-like continua are precisely the inverse limits of finite connected graphs with edge-contraction bonding maps, see \cite[Theorem~A and Theorem~14]{espinoza2020}.

\subsection{A universal tree-blowup for all graph-like continua}

In this section we will show that every graph-like continuum embeds topologically into the end compactification of a blowup $\mathbf{G}$ of the binary tree; in other words, this $|\mathbf{G}|$ is not only universal for all other end compactifications of locally finite connected graphs, as shown in Theorem~\ref{thm_universal1}, but even for the larger class of graph-like continua.

So let $T$ be the full rooted binary tree. 
For technical reasons, our desired universal graph $\mathbf{G}$ will be obtained by replacing every node $t \in T^n$ with a copy $K(t)$ of a somewhat larger complete graph $K_{2n+1}$ and joining, once again, each $K(t)$ completely to all $K(t')$ with $tt' \in E(T)$, as in Figure~\ref{fig_2n+1}.
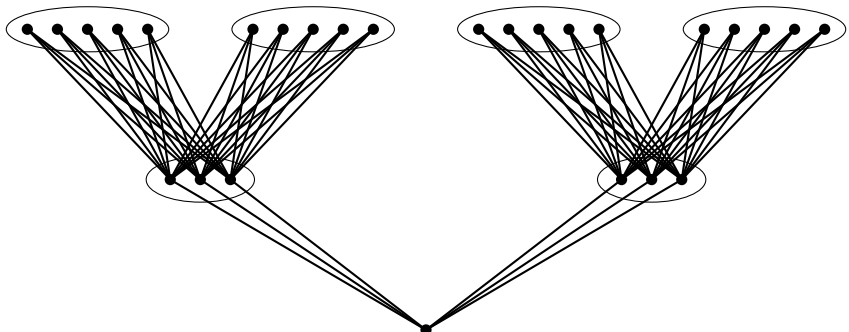
\begin{figure}[h]
    \centering
    \begin{tikzpicture}[scale=1]
     \foreach \point in {
	(3.0, 0),
	(5.6, 2),
	(6.0, 2),
	(6.3999999999999995, 2),
	(6.7, 4),
	(7.1000000000000005, 4),
	(7.500000000000001, 4),
	(7.8999999999999995, 4),
	(8.299999999999999, 4),
	(3.7, 4),
	(4.1000000000000005, 4),
	(4.5, 4),
	(4.9, 4),
	(5.3, 4),
	(-0.4, 2),
	(0.0, 2),
	(0.4, 2),
	(0.7, 4),
	(1.0999999999999999, 4),
	(1.4999999999999998, 4),
	(1.9000000000000001, 4),
	(2.3, 4),
	(-2.3, 4),
	(-1.9000000000000001, 4),
	(-1.5, 4),
	(-1.0999999999999999, 4),
	(-0.7, 4)
} {
            \def\this{vertex-\point}
            \node (\this) at \point  [circle,fill=black, inner sep=1.5pt] {} ;
        }
\draw (6.0, 2) ellipse (0.72cm and 0.3cm);
\draw (7.5, 4) ellipse (1.08cm and 0.3cm);
\draw (4.5, 4) ellipse (1.08cm and 0.3cm);
\draw (0.0, 2) ellipse (0.72cm and 0.3cm);
\draw (1.5, 4) ellipse (1.08cm and 0.3cm);
\draw (-1.5, 4) ellipse (1.08cm and 0.3cm);
\draw[thick] (3.0, 0)--(5.6, 2);
\draw[thick] (3.0, 0)--(6.0, 2);
\draw[thick] (3.0, 0)--(6.3999999999999995, 2);
\draw[thick] (5.6, 2)--(6.7, 4);
\draw[thick] (6.0, 2)--(6.7, 4);
\draw[thick] (6.3999999999999995, 2)--(6.7, 4);
\draw[thick] (5.6, 2)--(7.1000000000000005, 4);
\draw[thick] (6.0, 2)--(7.1000000000000005, 4);
\draw[thick] (6.3999999999999995, 2)--(7.1000000000000005, 4);
\draw[thick] (5.6, 2)--(7.500000000000001, 4);
\draw[thick] (6.0, 2)--(7.500000000000001, 4);
\draw[thick] (6.3999999999999995, 2)--(7.500000000000001, 4);
\draw[thick] (5.6, 2)--(7.8999999999999995, 4);
\draw[thick] (6.0, 2)--(7.8999999999999995, 4);
\draw[thick] (6.3999999999999995, 2)--(7.8999999999999995, 4);
\draw[thick] (5.6, 2)--(8.299999999999999, 4);
\draw[thick] (6.0, 2)--(8.299999999999999, 4);
\draw[thick] (6.3999999999999995, 2)--(8.299999999999999, 4);
\draw[thick] (5.6, 2)--(3.7, 4);
\draw[thick] (6.0, 2)--(3.7, 4);
\draw[thick] (6.3999999999999995, 2)--(3.7, 4);
\draw[thick] (5.6, 2)--(4.1000000000000005, 4);
\draw[thick] (6.0, 2)--(4.1000000000000005, 4);
\draw[thick] (6.3999999999999995, 2)--(4.1000000000000005, 4);
\draw[thick] (5.6, 2)--(4.5, 4);
\draw[thick] (6.0, 2)--(4.5, 4);
\draw[thick] (6.3999999999999995, 2)--(4.5, 4);
\draw[thick] (5.6, 2)--(4.9, 4);
\draw[thick] (6.0, 2)--(4.9, 4);
\draw[thick] (6.3999999999999995, 2)--(4.9, 4);
\draw[thick] (5.6, 2)--(5.3, 4);
\draw[thick] (6.0, 2)--(5.3, 4);
\draw[thick] (6.3999999999999995, 2)--(5.3, 4);
\draw[thick] (3.0, 0)--(-0.4, 2);
\draw[thick] (3.0, 0)--(0.0, 2);
\draw[thick] (3.0, 0)--(0.4, 2);
\draw[thick] (-0.4, 2)--(0.7, 4);
\draw[thick] (0.0, 2)--(0.7, 4);
\draw[thick] (0.4, 2)--(0.7, 4);
\draw[thick] (-0.4, 2)--(1.0999999999999999, 4);
\draw[thick] (0.0, 2)--(1.0999999999999999, 4);
\draw[thick] (0.4, 2)--(1.0999999999999999, 4);
\draw[thick] (-0.4, 2)--(1.4999999999999998, 4);
\draw[thick] (0.0, 2)--(1.4999999999999998, 4);
\draw[thick] (0.4, 2)--(1.4999999999999998, 4);
\draw[thick] (-0.4, 2)--(1.9000000000000001, 4);
\draw[thick] (0.0, 2)--(1.9000000000000001, 4);
\draw[thick] (0.4, 2)--(1.9000000000000001, 4);
\draw[thick] (-0.4, 2)--(2.3, 4);
\draw[thick] (0.0, 2)--(2.3, 4);
\draw[thick] (0.4, 2)--(2.3, 4);
\draw[thick] (-0.4, 2)--(-2.3, 4);
\draw[thick] (0.0, 2)--(-2.3, 4);
\draw[thick] (0.4, 2)--(-2.3, 4);
\draw[thick] (-0.4, 2)--(-1.9000000000000001, 4);
\draw[thick] (0.0, 2)--(-1.9000000000000001, 4);
\draw[thick] (0.4, 2)--(-1.9000000000000001, 4);
\draw[thick] (-0.4, 2)--(-1.5, 4);
\draw[thick] (0.0, 2)--(-1.5, 4);
\draw[thick] (0.4, 2)--(-1.5, 4);
\draw[thick] (-0.4, 2)--(-1.0999999999999999, 4);
\draw[thick] (0.0, 2)--(-1.0999999999999999, 4);
\draw[thick] (0.4, 2)--(-1.0999999999999999, 4);
\draw[thick] (-0.4, 2)--(-0.7, 4);
\draw[thick] (0.0, 2)--(-0.7, 4);
\draw[thick] (0.4, 2)--(-0.7, 4);
\end{tikzpicture}
    \vspace{-1cm}
    \caption{Tree Blowup. (Edges inside ellipses are omitted.) }
    \label{fig_2n+1}
\end{figure}

\begin{thm}
\label{thm_universalgraphlike}
The blowup $\mathbf{G}$ is universal for the class of graph-like continua.
\end{thm}

\begin{proof}
We will show that every graph-like continuum $(X,V,E)$ with vertex set $V$ and edge set $E$ embeds into $|\mathbf{G}|$ such that all vertices of $V$ are mapped to ends of $\mathbf{G}$, and all edges of $E$ are mapped to pairwise disjoint double rays in $\mathbf{G}$ between its corresponding ends.

By Theorem~\ref{thm_invlimit} we can represent $X$ as an inverse limit of finite connected multigraphs $\set{G_n}:{n \in \NN}$ where the bonding maps $f_n \colon G_{n+1} \to G_n$ correspond to the contraction of a single edge, i.e.\ where $G_{n+1} / \overline{e}_{n+1} = G_n$. We may assume that $G_0$ is the trivial graph on a single vertex without edges. 

We approximate the desired topological embedding $X \hookrightarrow |\mathbf{G}|$ combinatorially by two sequences of mappings $g_n$ and $p_n$ (for $n \in \NN$) with domain $V(G_n)$  and $E(G_n)$ respectively such that
\begin{enumerate}
    \item $g_n \colon V(G_n) \to T^n$ is injective,
    \item $p_n$ maps edges $e=xy$ of $E(G_n)$ to pairwise disjoint $K(g_n(x))-K(g_n(y))$ paths in $\mathbf{G}[T^{\leq n}]$, 
\end{enumerate}
and such that they are compatible in the following sense:
\begin{enumerate}[resume]
    \item $g_{n+1} (v)$ is a child of $g_n(f_n(v))$ for all $v \in V(G_{n+1})$, and
    \item $p_{n}(e) \subsetneq p_{n+1}(e)$ for all $e \in E(G_n)$
\end{enumerate}
The maps $g_n$ and $p_n$ are defined recursively. For $n=0$, the map $g_0$ sends the unique vertex of $G_0$ to the root of $T$, and $p_0 = \emptyset$ is fine.

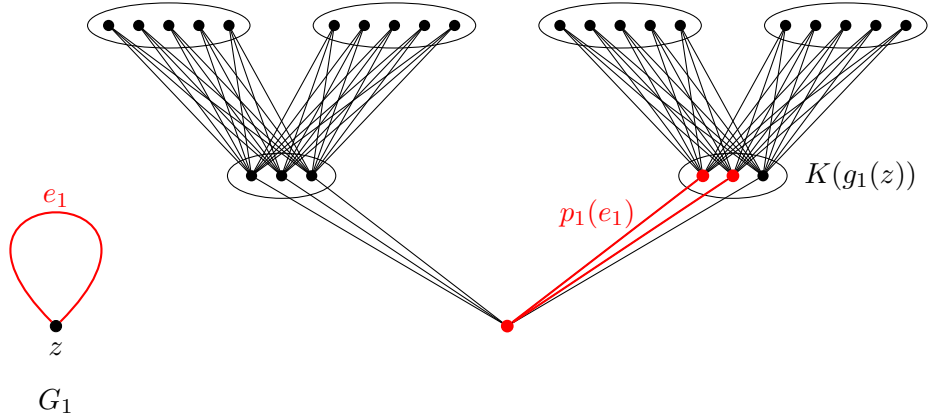
\begin{figure}[h]
    \centering
    \begin{tikzpicture}[scale=1]
     \foreach \point in {
	(3.0, 0),
	(5.6, 2),
	(6.0, 2),
	(6.3999999999999995, 2),
	(6.7, 4),
	(7.1000000000000005, 4),
	(7.500000000000001, 4),
	(7.8999999999999995, 4),
	(8.299999999999999, 4),
	(3.7, 4),
	(4.1000000000000005, 4),
	(4.5, 4),
	(4.9, 4),
	(5.3, 4),
	(-0.4, 2),
	(0.0, 2),
	(0.4, 2),
	(0.7, 4),
	(1.0999999999999999, 4),
	(1.4999999999999998, 4),
	(1.9000000000000001, 4),
	(2.3, 4),
	(-2.3, 4),
	(-1.9000000000000001, 4),
	(-1.5, 4),
	(-1.0999999999999999, 4),
	(-0.7, 4)
} {
            \def\this{vertex-\point}
            \node (\this) at \point  [circle,fill=black, inner sep=1.5pt] {} ;
        }

\draw (6.0, 2) ellipse (0.72cm and 0.3cm);
\draw (7.5, 4) ellipse (1.08cm and 0.3cm);
\draw (4.5, 4) ellipse (1.08cm and 0.3cm);
\draw (0.0, 2) ellipse (0.72cm and 0.3cm);
\draw (1.5, 4) ellipse (1.08cm and 0.3cm);
\draw (-1.5, 4) ellipse (1.08cm and 0.3cm);
\draw (3.0, 0)--(5.6, 2);
\draw (3.0, 0)--(6.0, 2);
\draw (3.0, 0)--(6.3999999999999995, 2);
\draw (5.6, 2)--(6.7, 4);
\draw (6.0, 2)--(6.7, 4);
\draw (6.3999999999999995, 2)--(6.7, 4);
\draw (5.6, 2)--(7.1000000000000005, 4);
\draw (6.0, 2)--(7.1000000000000005, 4);
\draw (6.3999999999999995, 2)--(7.1000000000000005, 4);
\draw (5.6, 2)--(7.500000000000001, 4);
\draw (6.0, 2)--(7.500000000000001, 4);
\draw (6.3999999999999995, 2)--(7.500000000000001, 4);
\draw (5.6, 2)--(7.8999999999999995, 4);
\draw (6.0, 2)--(7.8999999999999995, 4);
\draw (6.3999999999999995, 2)--(7.8999999999999995, 4);
\draw (5.6, 2)--(8.299999999999999, 4);
\draw (6.0, 2)--(8.299999999999999, 4);
\draw (6.3999999999999995, 2)--(8.299999999999999, 4);
\draw (5.6, 2)--(3.7, 4);
\draw (6.0, 2)--(3.7, 4);
\draw (6.3999999999999995, 2)--(3.7, 4);
\draw (5.6, 2)--(4.1000000000000005, 4);
\draw (6.0, 2)--(4.1000000000000005, 4);
\draw (6.3999999999999995, 2)--(4.1000000000000005, 4);
\draw (5.6, 2)--(4.5, 4);
\draw (6.0, 2)--(4.5, 4);
\draw (6.3999999999999995, 2)--(4.5, 4);
\draw (5.6, 2)--(4.9, 4);
\draw (6.0, 2)--(4.9, 4);
\draw (6.3999999999999995, 2)--(4.9, 4);
\draw (5.6, 2)--(5.3, 4);
\draw (6.0, 2)--(5.3, 4);
\draw (6.3999999999999995, 2)--(5.3, 4);
\draw (3.0, 0)--(-0.4, 2);
\draw (3.0, 0)--(0.0, 2);
\draw (3.0, 0)--(0.4, 2);
\draw (-0.4, 2)--(0.7, 4);
\draw (0.0, 2)--(0.7, 4);
\draw (0.4, 2)--(0.7, 4);
\draw (-0.4, 2)--(1.0999999999999999, 4);
\draw (0.0, 2)--(1.0999999999999999, 4);
\draw (0.4, 2)--(1.0999999999999999, 4);
\draw (-0.4, 2)--(1.4999999999999998, 4);
\draw (0.0, 2)--(1.4999999999999998, 4);
\draw (0.4, 2)--(1.4999999999999998, 4);
\draw (-0.4, 2)--(1.9000000000000001, 4);
\draw (0.0, 2)--(1.9000000000000001, 4);
\draw (0.4, 2)--(1.9000000000000001, 4);
\draw (-0.4, 2)--(2.3, 4);
\draw (0.0, 2)--(2.3, 4);
\draw (0.4, 2)--(2.3, 4);
\draw (-0.4, 2)--(-2.3, 4);
\draw (0.0, 2)--(-2.3, 4);
\draw (0.4, 2)--(-2.3, 4);
\draw (-0.4, 2)--(-1.9000000000000001, 4);
\draw (0.0, 2)--(-1.9000000000000001, 4);
\draw (0.4, 2)--(-1.9000000000000001, 4);
\draw (-0.4, 2)--(-1.5, 4);
\draw (0.0, 2)--(-1.5, 4);
\draw (0.4, 2)--(-1.5, 4);
\draw (-0.4, 2)--(-1.0999999999999999, 4);
\draw (0.0, 2)--(-1.0999999999999999, 4);
\draw (0.4, 2)--(-1.0999999999999999, 4);
\draw (-0.4, 2)--(-0.7, 4);
\draw (0.0, 2)--(-0.7, 4);
\draw (0.4, 2)--(-0.7, 4);

 \node (a) at (3.0, 0)  [circle,fill=red, inner sep=1.7pt] {} ;
        \node (b) at (5.6, 2)  [circle,fill=red, inner sep=1.7pt] {} ;
        \node (c) at (6.0, 2)  [circle,fill=red, inner sep=1.7pt] {} ;
        \draw[thick, red] (a) -- (b);
        \draw[thick, red] (a) -- (c);

 \node (z) at (-3, 0) [circle,fill=black, inner sep=1.5pt, label=below:$z$,draw] {} ;        
 \draw[thick, red] (z) .. controls (-5,2) and (-1,2) .. (z);
  \node (e1) at (-3, 1.7) {$\textcolor{red}{e_1}$} ;
   \node () at (-3, -1) {$G_1$} ;
  
    \node () at (4.2, 1.5) {$\textcolor{red}{p_1(e_1)}$} ;
    \node () at (7.7, 2) {$K(g_1(z))$} ;
 
\end{tikzpicture}
    \caption{Embeddings after $n=1$ steps.}
\end{figure}

Now suppose that $g_n$ and $p_n$ have been defined for some $n \in \NN$. 
First, we define $g_{n+1}$. By assumption, $G_{n+1}$ has one additional edge $e_{n+1} = vw$ where $G_{n+1} / \overline{e}_{n+1} = G_n$ and $f_n(e_{n+1}) = z$ is a vertex of $G_n$. If $e_{n+1}$ is a loop, we let $g_{n+1}(v) = g_{n+1}(w)$ be a child of $g_n(z)$ in $T$, and if $e_{n+1}$ is not a loop, we let $g_{n+1}(v) \neq g_{n+1}(w)$ be distinct children of $g_n(z)$ in $T$.  For all other vertices $x \in V(G_{n+1})$ we let $g_{n+1}(x)$ be an arbitrary child of $g_n(f_n(x))$ in $T$. Clearly, this assignment is injective and satisfies property (3).

Then we define $p_{n+1}$. Since $p_n(e)$ where paths in the graph $\mathbf{G}[T^{\leq n}]$, it is clear that we may extend $p_n(e)$ for all edges $e=xy \in E(G_n) \subset E(G_{n+1})$ by one additional edge at the front and back to obtain pairwise distinct $K(g_{n+1}(x))-K(g_{n+1}(y))$ paths in $\mathbf{G}[T^{\leq n+1}]$.
Now it remains to map the newly uncontracted edge $e_{n+1}=vw$. 
Since vertices of $G_n$ have degree at most $2n$, at most $2n$ vertices of $K(g_n(z))$ are end vertices of a path in the image of $p_n$, and so at least one vertex $\ell$ of $K(g_n(z))$ is still unused.
By the same reasoning at least three vertices are unused in every $K(g_{n+1}(u))$ with $u \in V(G_{n+1})$.
Now we pick two distinct unused vertices $a \in K(g_{n+1}(v))$ and $b \in K(g_{n+1}(w))$ and set $p_{n+1}(e_{n+1} )= a \ell b$.
This completes the recursive construction.

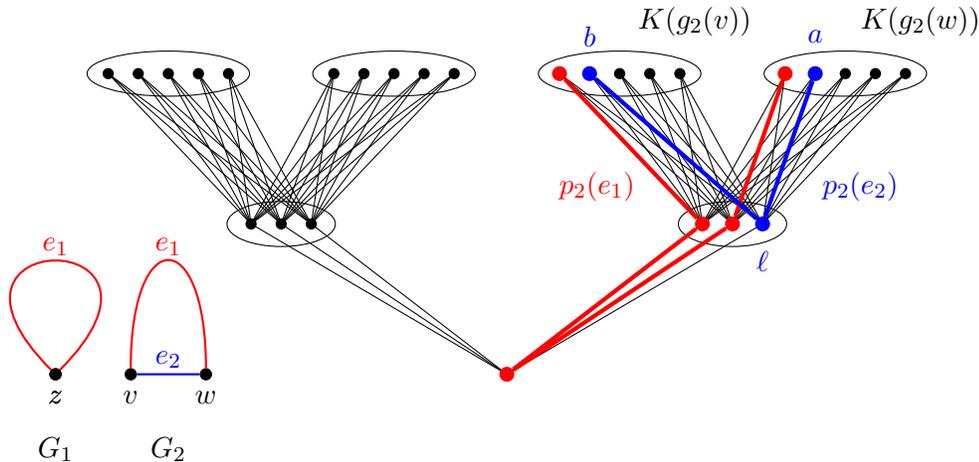
\begin{figure}[h]
    \centering
    \begin{tikzpicture}[scale=1]
     \foreach \point in {
	(3.0, 0),
	(5.6, 2),
	(6.0, 2),
	(6.3999999999999995, 2),
	(6.7, 4),
	(7.1000000000000005, 4),
	(7.500000000000001, 4),
	(7.8999999999999995, 4),
	(8.299999999999999, 4),
	(3.7, 4),
	(4.1000000000000005, 4),
	(4.5, 4),
	(4.9, 4),
	(5.3, 4),
	(-0.4, 2),
	(0.0, 2),
	(0.4, 2),
	(0.7, 4),
	(1.0999999999999999, 4),
	(1.4999999999999998, 4),
	(1.9000000000000001, 4),
	(2.3, 4),
	(-2.3, 4),
	(-1.9000000000000001, 4),
	(-1.5, 4),
	(-1.0999999999999999, 4),
	(-0.7, 4)
} {
            \def\this{vertex-\point}
            \node (\this) at \point  [circle,fill=black, inner sep=1.5pt] {} ;
        }

\draw (6.0, 2) ellipse (0.72cm and 0.3cm);
\draw (7.5, 4) ellipse (1.08cm and 0.3cm);
\draw (4.5, 4) ellipse (1.08cm and 0.3cm);
\draw (0.0, 2) ellipse (0.72cm and 0.3cm);
\draw (1.5, 4) ellipse (1.08cm and 0.3cm);
\draw (-1.5, 4) ellipse (1.08cm and 0.3cm);
\draw (3.0, 0)--(5.6, 2);
\draw (3.0, 0)--(6.0, 2);
\draw (3.0, 0)--(6.3999999999999995, 2);
\draw (5.6, 2)--(6.7, 4);
\draw (6.0, 2)--(6.7, 4);
\draw (6.3999999999999995, 2)--(6.7, 4);
\draw (5.6, 2)--(7.1000000000000005, 4);
\draw (6.0, 2)--(7.1000000000000005, 4);
\draw (6.3999999999999995, 2)--(7.1000000000000005, 4);
\draw (5.6, 2)--(7.500000000000001, 4);
\draw (6.0, 2)--(7.500000000000001, 4);
\draw (6.3999999999999995, 2)--(7.500000000000001, 4);
\draw (5.6, 2)--(7.8999999999999995, 4);
\draw (6.0, 2)--(7.8999999999999995, 4);
\draw (6.3999999999999995, 2)--(7.8999999999999995, 4);
\draw (5.6, 2)--(8.299999999999999, 4);
\draw (6.0, 2)--(8.299999999999999, 4);
\draw (6.3999999999999995, 2)--(8.299999999999999, 4);
\draw (5.6, 2)--(3.7, 4);
\draw (6.0, 2)--(3.7, 4);
\draw (6.3999999999999995, 2)--(3.7, 4);
\draw (5.6, 2)--(4.1000000000000005, 4);
\draw (6.0, 2)--(4.1000000000000005, 4);
\draw (6.3999999999999995, 2)--(4.1000000000000005, 4);
\draw (5.6, 2)--(4.5, 4);
\draw (6.0, 2)--(4.5, 4);
\draw (6.3999999999999995, 2)--(4.5, 4);
\draw (5.6, 2)--(4.9, 4);
\draw (6.0, 2)--(4.9, 4);
\draw (6.3999999999999995, 2)--(4.9, 4);
\draw (5.6, 2)--(5.3, 4);
\draw (6.0, 2)--(5.3, 4);
\draw (6.3999999999999995, 2)--(5.3, 4);
\draw (3.0, 0)--(-0.4, 2);
\draw (3.0, 0)--(0.0, 2);
\draw (3.0, 0)--(0.4, 2);
\draw (-0.4, 2)--(0.7, 4);
\draw (0.0, 2)--(0.7, 4);
\draw (0.4, 2)--(0.7, 4);
\draw (-0.4, 2)--(1.0999999999999999, 4);
\draw (0.0, 2)--(1.0999999999999999, 4);
\draw (0.4, 2)--(1.0999999999999999, 4);
\draw (-0.4, 2)--(1.4999999999999998, 4);
\draw (0.0, 2)--(1.4999999999999998, 4);
\draw (0.4, 2)--(1.4999999999999998, 4);
\draw (-0.4, 2)--(1.9000000000000001, 4);
\draw (0.0, 2)--(1.9000000000000001, 4);
\draw (0.4, 2)--(1.9000000000000001, 4);
\draw (-0.4, 2)--(2.3, 4);
\draw (0.0, 2)--(2.3, 4);
\draw (0.4, 2)--(2.3, 4);
\draw (-0.4, 2)--(-2.3, 4);
\draw (0.0, 2)--(-2.3, 4);
\draw (0.4, 2)--(-2.3, 4);
\draw (-0.4, 2)--(-1.9000000000000001, 4);
\draw (0.0, 2)--(-1.9000000000000001, 4);
\draw (0.4, 2)--(-1.9000000000000001, 4);
\draw (-0.4, 2)--(-1.5, 4);
\draw (0.0, 2)--(-1.5, 4);
\draw (0.4, 2)--(-1.5, 4);
\draw (-0.4, 2)--(-1.0999999999999999, 4);
\draw (0.0, 2)--(-1.0999999999999999, 4);
\draw (0.4, 2)--(-1.0999999999999999, 4);
\draw (-0.4, 2)--(-0.7, 4);
\draw (0.0, 2)--(-0.7, 4);
\draw (0.4, 2)--(-0.7, 4);

 \node (a) at (3.0, 0)  [circle,fill=red, inner sep=2pt] {} ;
        \node (b) at (5.6, 2)  [circle,fill=red, inner sep=2pt] {} ;
        \node (c) at (6.0, 2)  [circle,fill=red, inner sep=2pt] {} ;
        \draw[ultra thick, red] (a) -- (b);
        \draw[ultra thick, red] (a) -- (c);
        
         \node (b1) at (3.7, 4)  [circle,fill=red, inner sep=2pt] {} ;
        \node (c1) at (6.7, 4)  [circle,fill=red, inner sep=2pt] {} ;
        
        \draw[ultra thick, red] (b1) -- (b);
        \draw[ultra thick, red] (c1) -- (c);
        
          \node (x) at (7.1000000000000005, 4)  [circle,fill=blue, inner sep=2pt] {} ;
        \node (y) at (4.1000000000000005, 4)  [circle,fill=blue, inner sep=2pt] {} ;
        \node (ell) at (6.3999999999999995, 2)  [circle,fill=blue, inner sep=2pt] {} ;
        \draw[ultra thick, blue] (ell) -- (x);
        \draw[ultra thick, blue] (ell) -- (y);

 \node (z) at (-3, 0) [circle,fill=black, inner sep=1.5pt, label=below:$z$,draw] {} ;        
 \draw[thick, red] (z) .. controls (-5,2) and (-1,2) .. (z);
  \node (e1) at (-3, 1.7) {$\textcolor{red}{e_1}$} ;
   \node () at (-3, -1) {$G_1$} ;

    \node (z) at (-3, 0) [circle,fill=black, inner sep=1.5pt, label=below:$z$,draw] {} ;      
      \node (v) at (-2, 0) [circle,fill=black, inner sep=1.5pt, label=below:$v$,draw] {} ;      
        \node (w) at (-1, 0) [circle,fill=black, inner sep=1.5pt, label=below:$w$,draw] {} ;      
 \draw[thick, red] (v) .. controls (-2,2) and (-1,2) .. (w);
  \node (e1) at (-1.5, 1.7) {$\textcolor{red}{e_1}$} ;
  \draw[blue, thick] (v) -- (w);
  \node (e2) at (-1.5, 0.2) {$\textcolor{blue}{e_2}$} ;
   \node () at (-1.5, -1) {$G_2$} ;

    \node () at (4.2, 2.5) {$\textcolor{red}{p_2(e_1)}$} ;
      \node () at (7.7, 2.5) {$\textcolor{blue}{p_2(e_2)}$} ;
    \node () at (8.5, 4.7) {$K(g_2(w))$} ;
    \node () at (5.5, 4.7) {$K(g_2(v))$} ;
 
 \node () at (7.1, 4.5) {$\textcolor{blue}{a}$} ;
 \node () at (4.1, 4.5) {$\textcolor{blue}{b}$} ;
 \node () at (6.4, 1.5) {$\textcolor{blue}{\ell}$} ;
\end{tikzpicture}
    \caption{Embeddings after $n=2$ steps.}
\end{figure}

Next, observe that by property (4), for every edge $e=xy \in E$ the paths $(p_n(e))_{n \in \NN}$ give rise to pairwise disjoint double rays $p(e) = \bigcup_{n \in \NN} p_n(e)$ in $\mathbf{G}$. Fix an (order-preserving) homeomorphism $p_e \colon e \to p(e)$. 

For every $n \in \NN$ we consider the contractions $\mathbf{G}_n$ obtained by contracting all components of $\mathbf{G} - \mathbf{G}[T^{< n}]$, and write $q_n \colon \mathbf{G} \to \mathbf{G}_n$ for the corresponding contraction map. Note that every dummy vertex of $\mathbf{G}_n$ is of the form $v_t= \{ \blowup{T}_{\lfloor t \rfloor } \}$ for some $t \in T^n$. By Theorem~\ref{thm_limits} we have 
$|\mathbf{G}| = \varprojlim \mathbf{G}_n.$
By properties (1) and (2), the map
$$h_n \colon G_n \to \mathbf{G}_n, \; x \mapsto \begin{cases} q_n \circ p_e (x) & \text{ if } x \in e \in E(G_n) \\ v_{g_n(x)} & \text{ if } x \in V(G_n) \end{cases}$$
is a homeomorphic embedding such that the diagram
\begin{center}
\begin{tabular}{ccccccc}
$\mathbf{G}_1$ & $\xleftarrow{}$ & $\mathbf{G}_2$ & $\xleftarrow{}$ & $\mathbf{G}_3$ & $\xleftarrow{}$ & $\cdots$ \\
$\uparrow h_1$ &    & $\uparrow h_2$ &  & $\uparrow h_3$ &   &  \\ 
$G_1$ & $\xleftarrow{f_1}$ & $G_2$ & $\xleftarrow{f_2}$ & $G_3$ & $\xleftarrow{f_3}$ & $\cdots$
\end{tabular}
\end{center}
commutes.
By standard inverse limits arguments, see e.g.~\cite[2.22]{Nadler}, this yields a homeomorphic embedding of the respective inverse limits  $X = \varprojlim X_n $ into $|\mathbf{G}| = \varprojlim \mathbf{G}_n.$
\end{proof}

From this result, we re-obtain the following result from \cite[Theorem~15]{espinoza2020} as a corollary:

\begin{cor}
Every graph-like continuum can be embedded into the Freudenthal compactification of a locally finite graph.
\end{cor}

\bibliographystyle{plain}
\bibliography{ref}
\end{document}